\documentclass[a4paper, 11pt, reqno]{amsart}
\usepackage[utf8]{inputenc}
\usepackage[usenames,dvipsnames]{color}
\usepackage{amssymb, amsmath, amsthm, enumerate, latexsym, hyperref, mathtools}

\makeatother

\newcommand{\ol}{\overline}

\theoremstyle{plain}   \newtheorem{theorem}{Theorem}[section]
\theoremstyle{plain}   \newtheorem{proposition}[theorem]{Proposition}
\theoremstyle{plain}   \newtheorem{corollary}[theorem]{Corollary}
\theoremstyle{plain}   \newtheorem{lemma}[theorem]{Lemma}
\theoremstyle{plain}   \newtheorem{remark}[theorem]{Remark}
\theoremstyle{definition}   
\theoremstyle{definition}   
\theoremstyle{definition}   

\newcommand{\B}{\mathbb{B}}
\newcommand{\N}{\mathbb{N}}
\newcommand{\R}{\mathbb{R}}

\newcommand{\T}{\mathbb{T}}

\newcommand{\cA}{\mathcal{A}}
\newcommand{\cB}{\mathcal{B}}
\newcommand{\cC}{\mathcal{C}}

\newcommand{\cJ}{\mathcal{J}}

\newcommand{\cV}{\mathcal{V}}

\newcommand{\vJ}{\boldsymbol{\mathcal{J}}}

\DeclarePairedDelimiter{\parens}{\lparen}{\rparen}

\DeclarePairedDelimiterX{\pres}[2]{\langle}{\rangle}{#1\,\delimsize\vert\,\mathopen{}#2}

\newcommand{\Var}{\boldsymbol{\mathcal{V}}}

\newcommand*{\supp}[2][]{\mathrm{supp}\parens[#1]{#2}}

\newcommand*{\plac}{{\mathsf{plac}}}
\newcommand*{\placn}{{\mathsf{plac}_n}}

\newcommand*{\hypo}{{\mathsf{hypo}}}

\newcommand*{\sylv}{{\mathsf{sylv}}}

\newcommand*{\sylvh}{{\mathsf{sylv}^{\#}}}

\newcommand*{\baxt}{{\mathsf{baxt}}}

\newcommand*{\stal}{{\mathsf{stal}}}

\newcommand*{\taig}{{\mathsf{taig}}}

\newcommand*{\rPS}{{\mathsf{rPS}}}
\newcommand*{\rPSn}{{\mathsf{rPS}_n}}

\newcommand*{\styl}{{\mathsf{styl}}}
\newcommand*{\styln}{{\mathsf{styl}_n}}

\newcommand*{\kisn}{{\mathsf{Kis}_n}}

\newcommand*{\catn}{{\mathsf{Cat}_n}}

\newcommand*{\plit}{\mathrm{P}}

\newcommand*{\pplac}[2][]{\plit_{\plac}\parens[#1]{#2}}

\usepackage{tikz}
\usetikzlibrary{babel,cd,arrows.meta,calc,decorations.pathmorphing,decorations.text,matrix,positioning,shapes.geometric,shapes.misc,shapes.symbols}

\usetikzlibrary{matrix}

\tikzset{
	pretableaumatrix/.style={
		ampersand replacement=\&,
		matrix of math nodes,
		outer sep=1mm,
		inner sep=0mm,
		anchor=center,
		row sep={between borders,-\pgflinewidth},
		column sep={between borders,-\pgflinewidth},
		dottedentry/.style={densely dotted},
		spaceentry/.style={draw=none,execute at begin node=\null},
	},
	pretableaunode/.style={
		font=\small,
		draw=gray,
		sharp corners,
		rectangle,
		anchor=base,
		text height=3.75mm,
		text depth=1.25mm,
		minimum height=5mm,
		minimum width=5mm,
		inner sep=0mm,
		outer sep=0mm,
	},
	tableaumatrix/.style={
		pretableaumatrix,
		every node/.append style={
			pretableaunode,
		},
	},
	medtableaumatrix/.style={
		pretableaumatrix,
		every node/.append style={
			pretableaunode,
			font=\footnotesize,
			text height=2.75mm,
			text depth=.75mm,
			minimum height=3.5mm,
			minimum width=3.5mm
		},
	},
	smalltableaumatrix/.style={
		pretableaumatrix,
		every node/.append style={
			pretableaunode,
			font=\scriptsize,
			text height=1.85mm,
			text depth=.15mm,
			minimum height=2.5mm,
			minimum width=2.5mm,
		},
	},
	tinytableaumatrix/.style={
		pretableaumatrix,
		every node/.append style={
			pretableaunode,
			font=\tiny,
			text height=1.25mm,
			text depth=.15mm,
			minimum height=1.75mm,
			minimum width=1.75mm
		},
	},
	tableau/.style={
		baseline=-1.25mm,
		every matrix/.style={tableaumatrix},
	},
	medtableau/.style={
		baseline=-1.25mm,
		every matrix/.style={medtableaumatrix},
	},
	smalltableau/.style={
		baseline=-1.25mm,
		every matrix/.style={smalltableaumatrix},
	},
	preshapetableaumatrix/.style={
		pretableaumatrix,
		execute at end cell={\strut},
		every node/.append style={
			draw=black,
			anchor=base,
			inner sep=0mm,
			outer sep=0mm,
		},
		shadedentry/.style={fill=gray},
		darkshadedentry/.style={fill=darkgray},
	},
	medshapetableaumatrix/.style={
		preshapetableaumatrix,
		every node/.append style={
			text height=2.75mm,
			text depth=.75mm,
			minimum height=3.5mm,
			minimum width=3.5mm
		},
	},
	shapetableaumatrix/.style={
		ampersand replacement=\&,
		matrix of math nodes,
		outer sep=0mm,
		inner sep=0mm,
		anchor=base,
		row sep={between borders,-\pgflinewidth},
		column sep={between borders,-\pgflinewidth},
		execute at begin cell={\strut},
		every node/.append style={draw,anchor=base,text height=1mm,text depth=.5mm,minimum size=1.5mm,inner sep=0mm,outer sep=0mm},
	},
	shapetableau/.style={
		every matrix/.style={shapetableaumatrix},
	},
	topalign/.style={
		every matrix/.append style={name=maintableau,anchor=maintableau-1-1.base},
		baseline,
	},
}

\usetikzlibrary{shapes.geometric} 
\usetikzlibrary{shapes.misc}  

\tikzset{
	bst/.style={
		standard/.style={
			font=\small,
			draw=gray,
			rounded rectangle,
			minimum width=4.5mm,
			minimum height=4.5mm,
			inner xsep=0mm,
			inner ysep=1mm,
			outer sep=0mm,
			line width=.5pt,
		},
		empty/.style={
			minimum width=3mm,
			minimum height=3mm,
		},
		triangle/.style={
			isosceles triangle,
			isosceles triangle apex angle=60,
			shape border rotate=90,
			rounded corners=2mm,
			minimum width=8mm,
			inner xsep=0mm,
			inner ysep=.5mm
		},
		blank/.style={
			draw=none,
		},
		nodecount/.style={
			blank,
			font=\scriptsize,
		},
		every node/.style={standard},
		every child/.style={draw=black,line width=.6pt},
		level distance=10mm,
		level 1/.style={sibling distance=60mm},
		level 2/.style={sibling distance=30mm},
		level 3/.style={sibling distance=15mm},
	},
	medbst/.style={
		bst,
		level distance=10mm,
		level 1/.style={sibling distance=15mm},
		level 2/.style={sibling distance=15mm},
		level 3/.style={sibling distance=15mm},
	},
	smallbst/.style={
		bst,
		level distance=8mm,
		level 1/.style={sibling distance=10mm},
		level 2/.style={sibling distance=10mm},
		level 3/.style={sibling distance=10mm},
	},
	tinybst/.style={
		bst,
		level distance=5mm,
		level 1/.style={sibling distance=8mm},
		level 2/.style={sibling distance=8mm},
		level 3/.style={sibling distance=8mm},
		every node/.append style={
			font=\footnotesize,
		},
		triangle/.append style={
			rounded corners=1mm,
			minimum width=7mm,
			inner xsep=-.5mm,
		},
	},
	microbst/.style={
		bst,
		standard/.append style={
			font=\scriptsize,
			minimum width=3mm,
			minimum height=3mm,
			inner ysep=.25mm,
		},
		level distance=3mm,
		level 1/.style={sibling distance=6mm},
		level 2/.style={sibling distance=6mm},
		level 3/.style={sibling distance=6mm},
	},
	nanobst/.style={
		bst,
		standard/.append style={
			font=\tiny,
			minimum width=2mm,
			minimum height=2mm,
			inner ysep=.25mm,
		},
		level distance=2mm,
		level 1/.style={sibling distance=4mm},
		level 2/.style={sibling distance=4mm},
		level 3/.style={sibling distance=4mm},
	},
}
\title[Tropical Representations and Identities of the Stylic Monoid]{Tropical Representations and Identities of the Stylic Monoid}
\begin{document}
\maketitle
\begin{center}
THOMAS AIRD
\footnote{Department of Mathematics, University of Manchester. Email: \href{mailto:thomas.aird@manchester.ac.uk}{thomas.aird@manchester.ac.uk}.}
\footnote{This research was supported by a University of Manchester Dean's Scholarship Award.}
and DUARTE RIBEIRO
\footnote{Center for Computational and Stochastic Mathematics (CEMAT), University of Lisbon and Center for Mathematics and Applications (NovaMath), FCT NOVA and Department of Mathematics, FCT NOVA. Email: \href{mailto:dc.ribeiro@campus.fct.unl.pt}{dc.ribeiro@campus.fct.unl.pt}}
\footnote{This author's work was supported by National Funds through the FCT -- Funda\c{c}\~{a}o para a Ci\^{e}ncia e a Tecnologia, I.P., under the scope of the projects PTDC/MAT-PUR/31174/2017, UIDB/04621/2020 and UIDP/04621/2020 (Center for Computational and Stochastic Mathematics), and UIDB/00297/2020 and UIDP/00297/2020 (Center for Mathematics and Applications).}
\\
\end{center}

\begin{abstract}
    We exhibit a faithful representation of the stylic monoid of every finite rank as a monoid of upper unitriangular matrices over the tropical semiring. Thus, we show that the stylic monoid of finite rank $n$ generates the pseudovariety $\vJ_n$, which corresponds to the class of all piecewise testable languages of height $n$, in the framework of Eilenberg's correspondence. From this, we obtain the equational theory of the stylic monoids of finite rank, show that they are finitely based if and only if $n \leq 3$, and that their identity checking problem is decidable in linearithmic time. We also establish connections between the stylic monoids and other plactic-like monoids, and solve the finite basis problem for the stylic monoid with involution.
\end{abstract}


\section{Introduction}

Identities and varieties of semigroups and monoids have long been studied, and several important questions arise in this study, such as the question of whether a semigroup admits a finite basis for its equational theory. This question is known as the finite basis problem \cite{sapir_combinatorial,volkov_finitebasis}, and it is well-known that there are finite semigroups which are not finitely based \cite{perkins_bases}. Other questions regarding the variety generated by a semigroup are those of whether it contains only finitely generated subvarieties (see, for example, \cite{volkov_finitebasis}), or countably infinite subvarieties \cite{trakhtman_six}. We can also ask about the computational complexity of checking if an identity is satisfied by said semigroup. This problem is the identity checking problem \cite{kharlampovich_sapir_survey}, denoted by $\textsc{Check-Id}$, and is not only decidable for finite semigroups, but it is also known to be in the complexity class \textsf{coNP}. Due to Birkhoff's HSP Theorem \cite{birkhoff1935}, we can study these problems either by working directly with the identities, or by looking at homomorphic images, subsemigroups or direct products of other semigroups of which we already know enough to answer our questions. These problems have also been considered for involution semigroups, that is, semigroups equipped with a unary operation * which satisfies the identities $(x^*)^* \approx x$ and $(xy)^* \approx y^* x^*$ (see \cite{lee_thesis_2020} for a collection of results on this subject). In particular, the finite basis problem for finite involution semigroups has received much attention, since, contrary to intuition, finite involution semigroups and their underlying semigroups need not necessarily be simultaneously finitely based (see, for example, \cite{lee_involution_fb_reduct_not_fb,lee_involution_fb_reduct_not_fb}).

The tropical semiring $\T$, also known as the max plus semiring, is an important algebraic structure, due to its applications in many areas of mathematics, such as algebraic geometry and combinatorial optimization (see, for example, \cite{butkovic_max_linear_systems} and \cite{maclagan_sturmfels}). It is natural to consider matrices over this semiring which, under the induced operation of matrix multiplication, form monoids. These monoids have been used to great effect in representations of semigroups, in particular infinite semigroups which do not admit faithful finite dimensional representations over fields. For example, the famous bicyclic monoid has been shown to admit faithful representations over the tropical semiring \cite{daviaud_identities,izhakian_margolis_two_by_two}.

The plactic monoid $\plac$, whose elements can be identified with semistandard Young, has attracted the attention of many mathematicians, due to its connections with several areas of mathematics, such as algebraic combinatorics \cite{lothaire_2002}, representation theory \cite{green2006polynomial}, symmetric functions \cite{macdonald_symmetric} and crystal bases \cite{bump_crystalbases}. First studied by Schensted \cite{Schensted1961} and Knuth \cite{knuth1970}, and later studied in depth by Lascoux and Schützenberger \cite{LS1978}, the plactic monoid is now the focus of much attention with regards to its identities and varieties \cite{kubat_identities,izhakian_cloaktic}. 

Cain \textit{et al} \cite{ckkmo_placticidentity} have shown that the plactic monoid of finite rank $n$, denoted by $\placn$, does not satisfy any non-trivial identity of length less than or equal to $n$, hence, there is no single ``global'' identity satisfied by every plactic monoid of finite rank, and the infinite-rank plactic monoid does not satisfy any non-trivial identity. On the other hand, Johnson and Kambites \cite{johnson_kambites_tropical_plactic} have given a faithful representation of $\placn$ in monoids of upper triangular matrices over the tropical semiring $\T$, for every finite $n$. These monoids are known to satisfy non-trivial identities \cite{izhakian_identities,okninski_identities,taylor2017upper}, thus, Johnson and Kambite's result implies that every finite-rank plactic monoid satisfies non-trivial identities.

The study of the plactic monoid has given rise to a family of `plactic-like' monoids, whose elements can be uniquely identified with combinatorial objects. Monoids of this family include the hypoplactic monoid $\hypo$ \cite{novelli_hypoplactic,cm_hypo_crystal}; the sylvester monoid $\sylv$ and the \#-sylvester monoid $\sylvh$ \cite{hivert_sylvester,cm_sylv_crystal}; the Baxter monoid $\baxt$ \cite{giraudo_baxter2,cm_sylv_crystal}; the stalactic monoid $\stal$ \cite{hnt_com_hopf,priez_binary_trees}; the taiga monoid $\taig$ \cite{priez_binary_trees}; and the right patience-sorting monoid $\rPS$ \cite{cms_patience}. These monoids satisfy non-trivial identities \cite{cm_identities} and, except for the case of the right patience-sorting monoid, these identities are satisfied regardless of rank. The finite basis problem regarding the plactic-like monoids has also been studied by different authors (including the second author), using different techniques \cite{cain_malheiro_ribeiro_hypoplactic_2022,cain_malheiro_ribeiro_sylvester_baxter_2021,cain_johnson_kambites_malheiro_representations_plactic-like_2021,han2021preprint}, and the identity checking problem has been studied in \cite{cain_malheiro_ribeiro_sylvester_baxter_2021} by Cain, Malheiro and the second author.

The stylic monoid of finite rank $n$, introduced by Abram and Reutenauer in \cite{abram_reutenauer_stylic_monoid_2021} and denoted by $\styln$, is a finite quotient of the plactic monoid of rank $n$, defined by the action of words, over a finite totally ordered alphabet with $n$ letters, on the left of columns of semistandard Young tableaux, by Schensted left insertion. Its elements can be uniquely identified with so-called $N$-tableaux, and it is presented by the Knuth relations and the relations $x^2 \equiv x$, with $x \in \cA_n$. As such, to the author's knowledge, it is the first finite plactic-like monoid to be studied. It is a finite $\cJ$-trivial monoid (\cite[Theorem~11.1]{abram_reutenauer_stylic_monoid_2021}), hence, by \cite{simon_thesis_1972}, is in $\vJ_k$, the pseudovariety in Simon's hierarchy of $\cJ$-trivial monoids which corresponds to the class of all piecewise testable languages of height $k$, in Eilenberg's correspondence (\cite{Eilenberg_book76,pin1986varieties}), for some $k \in \N$. The pseudovariety $\vJ_k$ is defined by the equational theory $J_k$ of all identities $u \approx v$ such that $u$ and $v$ share the same subsequences of length $\leq k$. Blanchet-Sadri has studied these equational theories in depth (\cite{blanchet-sadri_equations_monoid_varieties_dot-depth_one_and_two_1994,blanchet-sadri_games_equations_dot-depth_hierarchy_1989,blanchet-sadri_equations_and_dot-depth_one_1993}), having shown that $J_k$ is finitely based if and only if $k \leq 3$. Johnson and Fenner, expanding upon Volkov's work \cite{volkov_reflexive_relations}, showed in \cite{johnson_fenner_unitriangular} that the variety described by $J_k$ is generated by the monoid $U_{k+1} (S)$ of $(k+1) \times (k+1)$ upper unitriangular matrices with entries in a non-trivial, idempotent commutative semiring $S$, of which the max-plus tropical semiring $\T$ is an example. 

In this work, we show that the stylic monoid of rank $n$ generates the pseudovariety $\vJ_n$. The paper is organized as follows: Section~\ref{section:background} gives the necessary background on the subject matter, namely words in Subsection~\ref{subsection:words}; identities and varieties in Subsection~\ref{subsection:identities_and_varieties}; semirings and matrix semigroups in Subsection~\ref{subsection:semirings_and_matrix_semigroups}; and the stylic monoid in Subsection~\ref{subsection:the_stylic_monoid}. In Section~\ref{section:tropical_representation_of_the_stylic_monoid}, we give a faithful representation of $\styln$ by $U_{n+1} (\T)$, thus proving that $\styln$ is in the variety generated by $U_{n+1} (\T)$, and we follow up in Section~\ref{section:main_result_and_conenctions_to_other_plactic_like_monoids} by showing that all identities satisfied by $\styln$ must also be in $J_n$, and therefore the equational theory of $\styln$ is $J_n$. From this, we deduce that the identity checking problem for $\styln$ is decidable in linearithmic time, and the variety generated by $\styln$, for $n \geq 3$, has uncountably many subvarieties. We also delve into the connections between the stylic monoid and other plactic-like monoids, by looking at the interactions between the varieties they generate. Finally, in Section~\ref{section:The_finite_basis_problem_for_the_stylic_monoid_with_involution}, we look at the finite basis problem for the stylic monoid with involution \textup{*} induced by the unique order-reversing permutation of $\cA_n$, and show that $(\styln, \textup{*})$ is finitely based if and only if $n = 1$. We also show that $(\styln, \textup{*})$ and $(U_{n+1} (\T), \,^\star)$, where $\,^\star$ is the skew transposition, do not generate the same variety for $n \geq 2$, which contrasts with the results obtained in Section~\ref{section:main_result_and_conenctions_to_other_plactic_like_monoids}.

The authors would like to note that, during the final revisions of this paper, Volkov \cite{volkov_stylic} proved the same result as Corollary~\ref{corollary:main_result} independently, by showing that $\styln$ is a homomorphic image of the Kiselman monoid $\kisn$ of rank $n$, and the Catalan monoid $\catn$ of rank $n$ is a homomorphic image of $\styln$. Both $\kisn$ and $\catn$ are known to generate the variety described by $J_n$ \cite[Theorem~8]{volkov_kiselman}.  


\section{Background}
\label{section:background}

We shall recall the necessary definitions and notations. For the necessary background on semigroups and monoids, see \cite{howie1995fundamentals}; for presentations, see \cite{higgins1992techniques}. Let $\N$ denote the set of natural numbers, without zero.


\subsection{Words}
\label{subsection:words}

Let $\Sigma$ be a non-empty set, referred to as an \emph{alphabet}, whose elements are refered to as \emph{letters}. Given an alphabet $\Sigma$, we denote the monoid of all words over $\Sigma$ under concatenation, the free monoid on $\Sigma$, by $\Sigma^*$, and we denote the empty word by $\varepsilon$. We write $|w|$ for the \emph{length} of a word $w$. For $1 \leq i \leq |w|$, we denote its $i$-th letter by $w_i$. The \textit{support} of $w$, denoted by $\supp{u}$, is the subset of $\Sigma$ of letters which occur in $w$.

For $u,v \in \Sigma^*$ we say that $u$ is a \emph{prefix} of $v$ if there exists $w' \in \Sigma^*$ such that $v = uw'$, and a \emph{subsequence} of $v$ if there exist $u_1, \dots, u_n \in \Sigma$ and $v'_1,\dots,v'_{n+1} \in \Sigma^*$ such that $u = u_1 \cdots u_n$ and $v =v'_1 u_1 v'_2 \cdots v'_n u_n v'_{n+1}$. We shall denote the prefix of the first $i$ letters of a word $w$ by $w_{\leq i}$, and denote the subsequence $u = u_1 \cdots u_n$ by its sequence of letters $u_1, \dots, u_n$.


\subsection{Identities and varieties}
\label{subsection:identities_and_varieties}

For a general background on universal algebra, see \cite{bs_universal_algebra}; on pseudovarieties, see \cite{almeida_1994_finite_semigroups}; on computation and complexity theory, see \cite{sipser2012introduction}. We also refer the reader to the survey \cite{volkov_finitebasis} on the finite basis problem for finite semigroups.

A \emph{monoid identity}, over an alphabet of variables $\Sigma$, is a formal equality $u \approx v$, where $u$ and $v$ are words in the free monoid $\Sigma^*$, and is \emph{non-trivial} if $u \neq v$. We say a variable $x$ \emph{occurs} in $u \approx v$ if $x$ occurs in $u$ or $v$. We say that a monoid $M$ \emph{satisfies} the identity $u \approx v$ if for every morphism $\psi: \Sigma^* \rightarrow M$ (also referred to as an \emph{evaluation}), the equality $\psi(u) = \psi(v)$ holds in $M$.

For a given monoid $M$, its \emph{identity checking problem} is the combinatorial decision problem $\textsc{Check-Id}{(M)}$ whose instance is an arbitrary identity $u \approx v$, and whose answer to such an instance is positive if $M$ satisfies the identity, and negative otherwise. As its input is only the identity and not the monoid, the time complexity of $\textsc{Check-Id}{(M)}$ should be measured only in terms of the size of the identity.

The class of all monoids which satisfies some set of identities $\Sigma$ is called an \emph{equational class}, and $\Sigma$ is called its \emph{equational theory}. An identity $u \approx v$ is a \emph{consequence} of a set of identities $\Sigma$ if all monoids which satisfy all identities of $\Sigma$ also satisfy $u \approx v$. An \emph{equational basis}, or simply basis, $\cB$ of an equational theory $\Sigma$ is a subset of $\Sigma$ such that each identity in $\Sigma$ is a consequence of $\cB$. We say an equational theory is \emph{finitely based} if it admits a finite basis, and \emph{non-finitely based} otherwise.

On the other hand, a class of monoids is a \emph{variety} if it is closed under taking homomorphic images, submonoids and direct products, while a class of finite monoids is a \emph{pseudovariety} if it is closed under taking homomorphic images, submonoids and finitary direct products. A \emph{subvariety} is a subclass of a variety which is itself a variety. We say a variety is \emph{generated} by a monoid $M$ if it is the smallest variety containing $M$, and denote it by $\Var(M)$. We define a pseudovariety \emph{generated} by a monoid in a parallel way.

A classical result by Birkhoff \cite{birkhoff1935} (in the context of monoids) states that a class of monoids is a variety if and only if it is an equational class. As such, a variety is uniquely determined by its equational theory. Therefore, for a given equational theory $\Sigma$, we denote the variety which corresponds to $\Sigma$ by $\Var(\Sigma)$.

An \emph{equational pseudovariety} is a pseudovariety which consists of all the finite monoids in some variety (see, for example, \cite{almeida_1994_finite_semigroups}). Though not all pseudovarieties are equational, every finitely generated pseudovariety is equational. An equational pseudovariety is defined by its equational theory. We say that a variety or an equational pseudovariety is finitely based if its equational theory is finitely based, and that a monoid is finitely based if the variety it generates is finitely based. 

For each $k \in \N$, we denote by $\vJ_k$ the pseudovariety which corresponds to the class of all piecewise testable languages of height $k$, by Eilenberg's correspondence, and by $J_k$ the set of all identities $u \approx v$ such that $u$ and $v$ share the same subsequences of length $\leq k$. The increasing sequence 
\[
\vJ_1 \subsetneq \vJ_2 \subsetneq \dots \subsetneq \vJ_k \subsetneq \dots,
\]
whose union is the pseudovariety $\vJ$ of all finite $\cJ$-trivial monoids, was introduced in \cite{simon_thesis_1972}, and is known as \emph{Simon's hierarchy of $\cJ$-trivial monoids}. Furthermore, a finite monoid is $\cJ$-trivial if and only if it is in $\vJ_k$ if and only if it satisfies all identities in $J_k$, for some $k$. Regarding whether these equational theories admit finite bases, we have the following:

\begin{enumerate}
    \item\label{basis_J_1} (\cite[folklore]{blanchet-sadri_equations_monoid_varieties_dot-depth_one_and_two_1994}) $J_1$ admits a finite basis, consisting of the following identities:
        \begin{align*}
        x^2 \approx x \quad &\text{and} \quad xy \approx yx.
        \end{align*}
    \item\label{basis_J_2} (\cite{simon_thesis_1972}) $J_2$ admits a finite basis, consisting of the following identities:
        \begin{align*}
        xyxzx \approx xyzx \quad &\text{and} \quad (xy)^2 \approx (yx)^2.
        \end{align*}
    \item\label{basis_J_3} (\cite[Proposition~4.1.6]{blanchet-sadri_games_equations_dot-depth_hierarchy_1989} and \cite{blanchet-sadri_equations_and_dot-depth_one_1993}) $J_3$ admits a finite basis, consisting of the following identities:
        \begin{align*}
        xyx^2zx &\approx xyxzx, \\
        xyzx^2tz &\approx xyxzx^2tx, \\
        zyx^2ztx &\approx zyx^2zxtx, \\
        (xy)^3 &\approx (yx)^3.
        \end{align*}
    \item (\cite[Theorem~3.4]{blanchet-sadri_equations_monoid_varieties_dot-depth_one_and_two_1994}) The equational theory $J_k$ is non-finitely based, for $k \geq 4$.
\end{enumerate}


\subsection{Semirings and matrix semigroups}
\label{subsection:semirings_and_matrix_semigroups}

For further reading on tropical semirings and their applications, see, for example, \cite{butkovic_max_linear_systems}.

A \emph{commutative semiring} $S$ is a set equipped with two binary operations $+$ and $\cdot$, such that $(S,+)$ and $(S,\cdot)$ are both commutative monoids, with respective neutral elements $0_S$ and $1_S$, that satisfy
\begin{align*}
    a \cdot (b+c) = a \cdot b + a \cdot c \quad \text{and} \quad 0_S \cdot a = 0_S,
\end{align*}
for all $a,b,c \in S$. We say $S$ is \emph{trivial} if $0_S = 1_S$, and $S$ is \emph{idempotent} if $a+a=a$, for all $a \in S$. The main relevant example of a non-trivial, idempotent commutative semiring for this paper is the \emph{tropical semifield} $\T := (\R \cup \{ -\infty \}, \oplus, \otimes)$, where $a \oplus b = \max(a,b)$ and $a \otimes b = a+b$, for $a,b \in \R \cup \{ -\infty \}$. Notice that $-\infty$ is the additive neutral element, or `zero', and $0$ is the multiplicative neutral element, or `one', of $\T$. Another relevant example is the \emph{Boolean semiring} $\B = \{0,1\}$, with maximum as addition and minimum as multiplication.

The set of all $n \times n$ matrices with entries in a commutative semiring $S$ forms a monoid under the matrix multiplication induced from the operations in $S$, and is denoted by $M_n (S)$. Its neutral element is the $n \times n$ identity matrix $I_{n \times n}$, whose diagonal entries are $1_S$ and all other entries are $0_S$, and its absorbing element is the $n \times n$ zero matrix $0_{n \times n}$, whose entries are all $0_S$. 
The submonoid of upper triangular matrices, that is, matrices whose entries below the main diagonal are $0_S$, of $M_n (S)$ is denoted by $UT_n (S)$, while the submonoid of upper unitriangular matrices, that is, upper triangular matrices whose diagonal entries are $1_S$, of $M_n (S)$ is denoted by $U_n (S)$.

Johnson and Fenner have shown in \cite[Corollary~3.3]{johnson_fenner_unitriangular} that, for $n \in \N$ and any non-trivial idempotent commutative semiring $S$, the monoid $U_{n+1} (S)$ generates $\Var(J_n)$, the variety whose equational theory is the set of identities $J_n$ defined in the previous subsection.


\subsection{The stylic monoid}
\label{subsection:the_stylic_monoid}

For a general background on the plactic monoid, see \cite[Chapter~5]{lothaire_2002}. For an in-depth look at the stylic monoid, see \cite{abram_reutenauer_stylic_monoid_2021}.

Let $\cA_n$ denote the totally ordered finite alphabet $\{ 1 < \cdots < n \}$. A \emph{semi-standard Young tableau}, or simply tableau, is a (finite) grid of cells, with down-left-aligned rows, filled with symbols from $\cA_n$, such that the entries in each row are weakly increasing from left to right, and the entries in each column are strictly decreasing from top to bottom. An example of a Young tableau is
\begin{center}  
	$\tikz[tableau]\matrix{
		5 \& 6 \\
		2 \& 3 \& 4 \\
		1 \& 1 \& 3 \& 3 \& 4 \\
	};$ .
\end{center}



Using \emph{Schensted's algorithm} \cite{Schensted1961}, we compute a unique tableau $\pplac{w}$, for each word over $\cA_n$. The \emph{plactic congruence} on $\cA_n^*$ is defined by
\[
u \equiv_\plac v \iff \pplac{u} = \pplac{v},
\]
for $u,v \in \cA_n^*$. The factor monoid $\cA_n^*/{\equiv_\plac}$ is the \emph{plactic monoid of rank} $n$, denoted by $\placn$ \cite{LS1978}. It follows from the definition of $\equiv_{\plac}$ that each element $[u]_{\plac}$ of $\plac$ can be identified with the Young tableau $\pplac{u}$. 

The plactic monoid of rank $n$ can also be defined by the presentation $\pres*{\cA_n}{\mathcal{R}_{\plac}}$ \cite{knuth1970}, where
\begin{align*}
	\mathcal{R}_{\plac} =& \left\{ (acb,cab): a \leq b < c \right\}\\
	& \cup \left\{ (bac,bca): a < b \leq c \right\}.
\end{align*}
The defining relations are known as the \emph{plactic relations}, or as the \emph{Knuth relations}. 

The \emph{stylic monoid of rank} $n$, denoted by $\styln$, is first defined in \cite[Section~5]{abram_reutenauer_stylic_monoid_2021} as the monoid of endofunctions of the set of columns over $\cA_n$ obtained by a left action of words on columns \cite[Section~4]{abram_reutenauer_stylic_monoid_2021}. It is a finite quotient of the free monoid over $\cA_n$, and the corresponding \emph{stylic congruence} of $\cA_n^*$ is denoted by $\equiv_{\styl}$. It is $\cJ$-trivial \cite[Theorem~11.1]{abram_reutenauer_stylic_monoid_2021}, therefore, by Simon's Theorem, there exists $k \in \N$ such that $\styln \in \vJ_k$. 

The stylic monoid of rank $n$ can be defined in two other ways, which will be the ones used in this work: It is defined by the presentation $\pres*{\cA_n}{\mathcal{R}_{\styl}}$ \cite[Theorem~8.1]{abram_reutenauer_stylic_monoid_2021}, where
\begin{align*}
	\mathcal{R}_{\styl} =& \mathcal{R}_{\plac} \cup \{ (a^2,a) : a \in \cA_n \}.
\end{align*}
The defining relations are known as the \emph{stylic relations}, and are the plactic relations together with a generator idempotent relation. As such, the stylic monoid of rank $n$ can be viewed as a quotient of the plactic monoid \cite[Proposition~5.1]{abram_reutenauer_stylic_monoid_2021}, and two words in the same stylic class have the same support \cite[Lemma~5.3]{abram_reutenauer_stylic_monoid_2021}.

For the other definition, we need a combinatorial object analogue to a Young tableau: An \emph{$N$-tableau} is a Young tableau where each row is strictly increasing and contained in the row below \cite[Subection~6.1]{abram_reutenauer_stylic_monoid_2021}. An example of an $N$-tableau is
\begin{center}  
	$\tikz[tableau]\matrix{
		5 \& 6 \\
		2 \& 5 \& 6 \\
		1 \&2 \& 3 \& 4 \& 5 \& 6 \\
	};$ .
\end{center}

As with Young tableaux and Schensted's algorithm, it is possible to associate each word $w \in \cA_n^*$ to a unique $N$-tableau, which we denote by $N(w)$, by using the \emph{right $N$-algorithm}: Consider rows of an $N$-tableau as subsets of the alphabet. The \emph{right $N$-insertion} of a letter $a \in \cA_n$ into a row $\cB \subseteq \cA_n$ gives the row $\cB \cup \{a\}$. If $b$ is the smallest letter in $B$ strictly greater than $a$, we say $b$ is \emph{bumped} (but $b$ is not deleted in $\cB \cup \{a\}$). The \emph{right $N$-insertion} of a letter $a \in \cA_n$ into an $N$-tableau is recursively defined as follows: $a$ is inserted into the first row, then, if a letter $b$ is bumped, $b$ is inserted into the row above. The algorithm stops when no letter is bumped. Inserting a letter into an $N$-tableau, using this algorithm, produces an $N$-tableau \cite[Proposition~6.1]{abram_reutenauer_stylic_monoid_2021}. The \emph{right $N$-insertion} of a word $w \in \cA_n^*$ into an $N$-tableau is done by inserting the letters of $w$, one-by-one from left-to-right. The stylic congruence on $\cA_n^*$ is defined by
\[
u \equiv_\styl v \iff N(u) = N(v),
\]
for $u,v \in \cA_n^*$ \cite[Theorem~7.1]{abram_reutenauer_stylic_monoid_2021}.

The stylic monoid of rank $n$ has an absorbing element, which is the stylic class of the decreasing product of all letters in $\cA_n^*$ \cite[Proposition~5.4]{abram_reutenauer_stylic_monoid_2021}. 
This element corresponds to the $N$-tableau with $n$ rows and the letters $\{i,\dots,n\}$ in the $i$-th row.

The following definitions are introduced in \cite[Subsection~6.3]{abram_reutenauer_stylic_monoid_2021}: For each subset $\cB$ of $\cA_n$, and each letter $a \in \cA_n$, the element $a_{\cB}^{\uparrow} \in \cB \cup \{\varepsilon\}$ is the smallest letter in $\cB$ which is strictly greater than $a$, or $\varepsilon$ if such a letter does not exist. Define the mapping $\delta: \cA_n^* \to \cA_n^*$ as follows: for any word $w \in \cA_n^*$ and letter $a \in \cA_n$, $\delta(wa) = \delta(w) \cdot a_{\supp{w}}^{\uparrow}$. Notice that the smallest letter in $w$ is not in $\delta(w)$, hence $\supp{\delta^k (w)} \subsetneq \supp{\delta^{k-1} (w)}$, for all $k \in \N$ such that $\supp{\delta^{k-1} (w)} \neq \emptyset$.

We introduce the following definition, which expands upon the ``arrow'' notation: For a word $w \in \cA_n^*$, and $k \in \N$, define the mapping ${\uparrow}^k_w : \{1,\dots,|w|\} \rightarrow \supp{w}$ recursively, as follows: for $2 \leq l \leq k$,
\begin{align*}
   {{\uparrow}^1_w}{(i)} &= {(w_{i})}_{\supp{w_{\leq i}}}^{\uparrow},\\
   {{\uparrow}^l_w}{(i)} &= {\left({{\uparrow}^{l-1}_w}{(i)}\right)}_{\supp{\delta^{l-1}(w_{\leq i})}}^{\uparrow}.
\end{align*}
If ${{\uparrow}^{k}_w}{(i)} \neq \varepsilon$, then ${{\uparrow}^{k}_w}{(i)}$ is the letter which is bumped into the $(k+1)$-th row when $w_i$ is inserted into the $N$-tableau. As an example, consider the word $535234512345$. Then,
\[
\begin{array}{cccccccccccccl}
    5 & 3 & 5 & 2 & 3 & 4 & 5 & 1 & 2 & 3 & 4 & 5 & = & w,\\
      & 5 &   & 3 & 5 & 5 &   & 2 & 3 & 4 & 5 &   & = & \delta(w),\\
      &   &   & 5 &   &   &   & 3 & 5 & 5 &   &   & = & \delta^2(w),\\
      &   &   &   &   &   &   & 5 &   &   &   &   & = & \delta^3(w),
\end{array}
\]
and ${{\uparrow}^3_w}{(8)} = 5$, that is, the letter $w_8 = 1$ bumps $5$ to the third row of $N(535234512345)$.

The following lemmata are immediate consequences of the definition of ${\uparrow}^k_w$, and the right $N$-algorithm in the case of the first lemma:

\begin{lemma} \label{lemma:row_bump}
    Let $w \in \cA_n^*$, $a \in \cA_n$, and $k \in \N$. Then, $a$ occurs in the $k$-th row of $N(w)$ if and only if there exists an index $j \leq |w|$ such that ${{\uparrow}^{k-1}_w}{(j)} = a$.
\end{lemma}
\begin{proof}
By repeated application of \cite[Lemma~6.3]{abram_reutenauer_stylic_monoid_2021}, $\supp{\delta^{k-1}(w)}$ is the $k$-th row of $N(w)$, viewed as a subset of $\cA_n$. Moreover, by the definition of ${\uparrow}^{k-1}_w$, for $a \in \cA_n$, we have that $a \in \supp{\delta^{k-1}(w)}$ if and only if ${{\uparrow}^{k-1}_w}{(j)} = a$ for some $j \leq |w|$. Thus, $a$ is in the $k$-th row of $N(w)$ if and only if there is some $j$ satisfying the previously mentioned condition.
\end{proof}

\begin{lemma} \label{lemma:up_arrow_to_decreasing_subsequence}
Let $w \in \cA_n^*$ and let $k, s_k \in \N$ be such that $1 \leq k \leq s_k \leq |w|$. If ${{\uparrow}^{k-1}_w}{(s_k)} = a$, for some $a \in \cA_n$, then there exists a strictly decreasing subsequence $w_{s_1},\dots,w_{s_k}$ of $w$ such that $w_{s_1} = a$ and ${{\uparrow}^{l-1}_w}{(s_l)} = a$ for $1 < l \leq k$.
\end{lemma}
\begin{proof}
Since $\supp{\delta^l (w)} \subsetneq \supp{\delta^{l-1} (w)}$, for all $l < k$, then ${{\uparrow}^{k-1}_w}{(s_k)} = a$ implies that $a \in \supp{\delta^l(w)}$, for all $1 \leq l \leq k-1$, and $a \in \supp{w}$. Thus, there exist $1 \leq s_1, \dots, s_{k-1} \leq |w|$ such that $w_{s_1} = a$ and ${{\uparrow}^{l-1}_w}{(s_l)} = a$ for $1 < l \leq k-1$. 

\par Notice that ${{\uparrow}^{l-1}_w}{(s_l)} = a$ implies that there is a letter $a$ to the left of ${{\uparrow}^{l-2}_w}{(s_l)}$ in $\delta^{l-2}(w)$, for all $2 < l \leq k$. Similarly, ${{\uparrow}^{1}_w}{(s_2)} = a$ implies that there is a letter $a$ to the left of $w_{s_2}$ in $w$. As such, we can restrict the choice of $s_1, \dots, s_{k-1}$ to have $s_1 < \dots < s_k$.

Furthermore, notice that, since ${{\uparrow}^{l-1}_w}{(s_l)} = a$, there must exist $i \leq s_l$ such that ${{\uparrow}^{l-2}_w}{(i)} = a$ and $w_i > w_{s_l}$: In order to obtain a contradiction, take $i$ such that $w_i \leq w_{s_l}$,
\[
{{\uparrow}^{j}_w}{(i)} \leq {{\uparrow}^{j}_w}{(s_l)} <  {{\uparrow}^{j+1}_w}{(s_l)} < {{\uparrow}^{j+1}_w}{(i)},
\]
and ${{\uparrow}^{j'}_w}{(i)} \leq {{\uparrow}^{j'}_w}{(s_l)}$ for all $1 \leq j' \leq j$, such that $j$ is minimal. In other words, when comparing the sequences of ``arrows'' of $i$ and $s_l$, this choice of $i$ gives us the sequence where there are the least number of elements which are less than or equal to the corresponding elements of the sequence of $s_l$, i.e. 
\[
    \begin{array}{ccc}
        w_i & \leq & w_{s_l} \\
        {{\uparrow}^{1}_w}{(i)} & \leq & {{\uparrow}^{1}_w}{(s_l)} \\
        \vdots & & \vdots \\
        {{\uparrow}^{j}_w}{(i)} & \leq & {{\uparrow}^{j}_w}{(s_l)} \\
        {{\uparrow}^{j+1}_w}{(i)} & > & {{\uparrow}^{j+1}_w}{(s_l)} \\
        \vdots & & \vdots \\
        {{\uparrow}^{l-2}_w}{(i)} & > & {{\uparrow}^{l-2}_w}{(s_l)} \\
         & & {{\uparrow}^{l-1}_w}{(s_l)}
    \end{array}
\]
Then, we have that all occurrences of ${{\uparrow}^{j+1}_w}{(s_l)}$ must be to the right of ${{\uparrow}^{j}_w}{(i)}$ in $\delta^j(w_{\leq s_l})$, since ${{\uparrow}^{j}_w}{(i)}$ bumps ${{\uparrow}^{j+1}_w}{(i)}$ and not ${{\uparrow}^{j+1}_w}{(s_l)}$. But at least one occurrence of ${{\uparrow}^{j+1}_w}{(s_l)}$ in $\delta^j(w_{\leq s_l})$ will bump $a$ to the $(l-2)$-th row. This contradicts the minimality of $j$, hence, we can choose $s_1, \dots, s_{l}$ such that $s_1 < \dots < s_k$ and $w_{s_1} > \dots > w_{s_l}$. 

Thus, we have found a strictly decreasing subsequence $w_{s_1},\dots,w_{s_k}$ of $w$, where $w_{s_1} = a$ and ${{\uparrow}^{l-1}_w}{(s_l)} = w_{s_1}$ for all $1 < l \leq k$.
\end{proof}


\section{Tropical Representations of the Stylic Monoid}
\label{section:tropical_representation_of_the_stylic_monoid}

We first construct a faithful representation of the stylic monoid of finite rank $n$ on the monoid of upper unitriangular $(n+1) \times (n+1)$ tropical matrices, for each $n \in \N$. Since, by \cite[Corollary~3.3]{johnson_fenner_unitriangular}, this monoid generates the variety with equational theory $J_n$, we show that $\styln$ satisfies all identities in $J_n$. 

Let $\ol{x} := n+1 - x$ for all $x \in \cA_n$. We define the map $\rho_n \colon \cA_n^* \rightarrow U_{n+1}(\T)$ as follows:
\[ \rho_n(x)_{i,j}  = 
\begin{cases}
0 &\text{if } i = j; \\
1 &\text{if } i \leq \ol{x} < j; \\
-\infty &\text{otherwise}.
\end{cases}
\]
for each $x \in \cA_n$, extending multiplicatively to all of $\cA_n^*$ and defining $\rho_n(\varepsilon) = I_{(n+1) \times (n+1)}$. For example, the images of $2$ and of $4213$ under $\rho_4$ are, respectively,
\[
    \begin{bmatrix}
        0 & -\infty & -\infty & 1 & 1 \\
        -\infty & 0 & -\infty & 1 & 1 \\
        -\infty & -\infty & 0 & 1 & 1 \\
        -\infty & -\infty & -\infty & 0 & -\infty \\
        -\infty & -\infty & -\infty & -\infty & 0
    \end{bmatrix}
    \quad \text{and} \quad
    \begin{bmatrix}
        0 & 1 & 2 & 2 & 3 \\
        -\infty & 0 & 1 & 1 & 2 \\
        -\infty & -\infty & 0 & 1 & 2 \\
        -\infty & -\infty & -\infty & 0 & 1 \\
        -\infty & -\infty & -\infty & -\infty & 0
    \end{bmatrix}
\]


Notice that, for each $x \in \cA_n$, its image under $\rho_n$ is a unitriangular tropical matrix where the only entries above the diagonal different from $-\infty$ are equal to $1$.

\begin{lemma} \label{lemma:maximum_length_strictly_decreasing}
Let $w \in \cA_n^*$. For $1 \leq i < j \leq n+1$ and $k \in \N$, we have that $\rho_n(w)_{i,j} = k$ if and only if $k$ is the maximum length of any strictly decreasing subsequence of $w$ only using letters between $\ol{j}+1$ and $\ol{i}$. On the other hand, $\rho_n(w)_{i,j} = -\infty$ if and only if $w$ does not contain $a$ for any $i \leq \ol{a} < j$. 
\end{lemma}

A remark about abuse of language: we say ``only using letters between $\ol{j}+1$ and $\ol{i}$'' in order to avoid the formal, but more cumbersome, statement ``only using letters $a \in \cA_n$ such that $\ol{j}+1 \leq a \leq \ol{i}$''.

\begin{proof}

Let $w \in \cA_n^*$ and $1 \leq i < j \leq n+1$. Suppose $\rho_n(w)_{i,j} = k$, for some $k \in \N$. Then, by the definition of tropical matrix multiplication, $w$ admits a subsequence $w_{s_1}, \dots, w_{s_k}$, of length $k$, and there exist $i = t_0 < \cdots < t_k = j$ such that $\rho_n(w_{s_i})_{t_{i-1},t_i} = 1$ for all $1 \leq i \leq k$. Furthermore, by the definition of $\rho_n$, $t_{i-1} \leq \ol{w_{s_i}} < t_i$. Therefore, $w_{s_1},\dots,w_{s_k}$ is a strictly decreasing subsequence of $w$ such that $\ol{i} \geq w_{s_1} > \dots > w_{s_k} \geq \ol{j}+1$ and hence, the maximum length of a strictly decreasing subsequence of $w$ only using letters between $\ol{j}+1$ and $\ol{i}$ is greater than or equal to $\rho_n(w)_{i,j}$.
\par Suppose now that $k$ is the maximum length of any strictly decreasing subsequence of $w$ only using letters between $\ol{j}+1$ and $\ol{i}$. Let $w_{s_1},\dots,w_{s_k}$ be a strictly decreasing subsequence of $w$ such that $\ol{i} \geq w_{s_1} > \dots > w_{s_k} \geq \ol{j}+1$, then let $t_0 = i, t_k = j$ and $t_i = \ol{w_{s_{i+1}}}$ for $ 1 \leq i < k$. Hence, by the definition of $\rho_n$, we have that $\rho_n(w_{s_i})_{t_{i-1},t_{i}} = 1$ for $ 1 \leq i \leq k$, and therefore
\[\rho_n(w)_{i,j} \geq \prod_{i=1}^{k} \rho_n(w_{s_i})_{t_{i-1},t_{i}} = k. \]
Thus, $\rho_n(w)_{i,j}$ is greater than or equal to the maximum length of a strictly decreasing subsequence only using letters between $\ol{j}+1$ and $\ol{i}$. Equality follows.
\par In the case where $\rho_n(w)_{i,j} = -\infty$, there is no $t \in \{1, \dots, |w|\}$ such that $i \leq \ol{w_t} < j$, otherwise, $w_t$ would form a strictly decreasing subsequence of $w$ (with just one letter), only using letters between $\ol{j}+1$ and $\ol{i}$, which would imply that $\rho_n(w)_{i,j} \geq 1$. Similarly, if $\ol{w_t} < i$ or $\ol{w_t} \geq j$ for all $1 \leq t \leq |w|$, then $\rho_n(w_t)_{i,j'} = -\infty$ for all $i < j' \leq j$ and hence $\rho_n(w)_{i,j} = -\infty$.
\end{proof}

As an immediate corollary, notice that $\rho_n(w)_{i,j} \leq n$, for all $1 \leq i,j \leq n+1$. We also have the following:

\begin{corollary}
\label{corollary:difference_of_one_in_finite_adjacent_entries_of_matrix}
Let $w \in \cA_n^*$. Then, any two finite adjacent entries in $\rho_n(w)$ must differ by at most $1$, and are weakly increasing on columns and weakly decreasing on rows. 
In other words, for $1 \leq i \leq j \leq n+1$, if $\rho_n(w)_{i,j}$ and $\rho_n(w)_{i+1,j}$ are both finite, then $\rho_n(w)_{i+1,j} \leq \rho_n(w)_{i,j} \leq \rho_n(w)_{i+1,j} + 1$. Similarly, if $\rho_n(w)_{i,j}$ and $\rho_n(w)_{i,j+1}$ are both finite, then $\rho_n(w)_{i,j} \leq \rho_n(w)_{i,j+1} \leq \rho_n(w)_{i,j} + 1$.
\end{corollary}

\begin{proof}
\par First, by noticing that any strictly decreasing subsequence only using letters between $\ol{j}+1$ and $\ol{i}$ is also a strictly decreasing subsequence only using letters between $\ol{j}$ and $\ol{i}$, and $\ol{j}+1$ and $\ol{i}+1$, we have that the the entries of $\rho_n(w)$ weakly increase left-to-right on the columns and weakly decrease top-to-bottom on the rows.
\par Suppose, in order to obtain a contradiction, that there exist $1 \leq i \leq j \leq n$ and $0 \leq k < k' \leq n$ such that $\rho_n(w)_{i,j} = k$ and $\rho_n(w)_{i,j+1} = k'+1$. By the previous lemma, there are maximum length strictly decreasing subsequences $u$ and $v$ of $w$, of length $k$ and $k'+1$ and only using letters between $\ol{j}+1$ and $\ol{i}$ and between $\ol{j}$ and $\ol{i}$, respectively. Taking $v$ and discarding its smallest letter gives us a strictly decreasing subsequence of $w$, of length $k'$, only using letters between $\ol{j}+1$ and $\ol{i}$, which contradicts the maximality of the length of $u$. Similarly, we can prove that there are no $2 \leq i \leq j \leq n+1$ and $1 \leq k < k' \leq n$ such that $\rho_n(w)_{i,j} = k$ and $\rho_n(w)_{i-1,j} = k'+1$. 
\end{proof}

\begin{proposition}
\label{proposition:well_defined_morphism}
The map $\rho_n$ induces a well-defined morphism from $\styln$ to $U_{n+1}(\T)$.
\end{proposition}
\begin{proof}
We show that $\rho_n$ satisfies the stylic relations, that is $x^2 \equiv x$ for all $x \in \cA_n$ and the Knuth relations.
\par To show that $\rho_n(x^2) = \rho_n(x)$ for all $x \in \cA_n$, begin by observing that for all $i \leq j$, $\rho_n(x^2)_{i,j} = \rho_n(x)_{i,k} \cdot \rho_n(x)_{k,j}$ for some $i \leq k \leq j$. Suppose $\rho_n(x^2)_{i,j} \neq -\infty$. If there exists $i < k < j$ such that $\rho_n(x)_{i,k} \neq -\infty \neq \rho_n(x)_{k,j}$, then we have that $i \leq \ol{x} < k \leq \ol{x} < j$, giving a contradiction. Thus, we either have $i = k$ or $k = j$. In either case, as $\rho_n(x)_{i,i} = \rho_n(x)_{j,j} = 0$, we have that $\rho_n(x^2)_{i,j} = \rho_n(x)_{i,j}$. If $\rho_n(x^2)_{i,j} = -\infty$, then as $\rho_n(x^2)_{i,j} \geq \rho_n(x)_{i,j} \cdot \rho_n(x)_{j,j}$, we have that $\rho_n(x)_{i,j} = -\infty$.
\par For the Knuth relations, both sides of each relation have the same number of occurrences of each letter, and are of length $3$. Let $w$ be one side of a Knuth relation, then by Lemma~\ref{lemma:maximum_length_strictly_decreasing}, $\rho_n(w)_{i,j} \in \{-\infty,0,1,2\}$ for all $i,j$, as $w$ does not contain a strictly decreasing subsequence of length 3.
Moreover, it is clear to see that $\rho_n(w)_{i,j} = 0$ if and only if $i = j$.
\par Let $u \equiv v$ be a Knuth relation. Then, $\rho_n(u)_{i,j} \neq -\infty$ if and only if $i = j$ or $i \leq \ol{u_t} < j$ for some $t \in \{1,2,3\}$. Thus, as $u$ and $v$ have the same content, $\rho_n(u)_{i,j} \neq -\infty$ if and only if $\rho_n(v)_{i,j} \neq -\infty$.
\par Finally, it suffices to show that $\rho_n(u)_{i,j} = 2$ if and only if $\rho_n(v)_{i,j} = 2$. Observe that, as $\rho_n(u)_{i,j} \leq 2$, then $\rho_n(u)_{i,j} = 2$ if and only if there exists $i \leq k \leq j$ such that $\rho_n(u_{s_1})_{i,k} = \rho_n(u_{s_2})_{k,j} = 1$ for some $1 \leq s_1 < s_2 \leq 3$ and hence, $i \leq \ol{u_{s_2}} < k \leq \ol{u_{s_1}} \leq j$.
\par By considering all the decreasing sequences in both sides of each Knuth relation, it suffices to show that if $\rho_n(ca)_{i,j} = 2$ then $\rho_n(ba)_{i,j} = 2$ for $a < b \leq c$ and $\rho_n(cb)_{i,j} = 2$ for any $a \leq b < c$.
\par Suppose $\rho_n(ca)_{i,j} = 2$ for $a < b \leq c$. Then, there exists $k$ such that $\rho_n(c)_{i,k} = \rho_n(a)_{k,j} = 1$, with $i \leq \ol{c} < k \leq \ol{a} < j$. But then as $a < b \leq c$, there exists $k'$ such that $i \leq \ol{b} < k' \leq \ol{a} < j$, hence $\rho_n(b)_{i,k'} = \rho_n(a)_{k',j} = 1$. Similarly, if $\rho_n(ca)_{i,j} = 2$ for $a \leq b < c$ then there exists $k$ such that $\rho_n(c)_{i,k} = \rho_n(a)_{k,j} = 1$, with $i \leq \ol{c} < k \leq \ol{a} < j$. But then as $a \leq b < c$, there exists $k'$ such that $i \leq \ol{c} < k' \leq \ol{b} < j$, hence $\rho_n(c)_{i,k'} = \rho_n(b)_{k',j} = 1$. Thus, $\rho_n$ respects the Knuth relations.
\end{proof}

Let us denote by $\varrho_n$ the induced morphism from $\styln$ to $U_{n+1}(\T)$. For example, the words $4213$, $4214234$ and $4241234$ are in the same stylic class, and the image of $[4213]_{\styl_4}$ under $\varrho_4$ is the same as that of $4213$ under $\rho_4$, that is,
\[
    \tikz[tableau]\matrix{
		4 \\
		2 \& 4 \\
		1 \&2 \& 3 \& 4 \\
	};
	\quad \xmapsto{\quad \varrho_4 \quad} \quad
    \begin{bmatrix}
        0 & 1 & 2 & 2 & 3 \\
        -\infty & 0 & 1 & 1 & 2 \\
        -\infty & -\infty & 0 & 1 & 2 \\
        -\infty & -\infty & -\infty & 0 & 1 \\
        -\infty & -\infty & -\infty & -\infty & 0
    \end{bmatrix}
\]

The following lemma allows us to deduce if a letter $a$ occurs in the $k$-th row of $N(w)$, by looking at the image of $N(w)$ under $\rho_n$ and seeing if, in line $\overline{a}$, the leftmost entry with value $k$ (if it exists) has below it an entry with value $k-1$:
\begin{lemma}
    Let $w \in \cA_n^*$, $a \in \cA_n$, and $k \in \N$. Then, $a$ occurs in the $k$-th row of $N(w)$ if and only if there exists $j \in \{1, \dots, n+1\}$, with $\overline{a} < j$, such that $\rho_n(w)_{\overline{a},j} = k$, and $\rho_n(w)_{\overline{a}+1,j} = k-1$.
\end{lemma}
\begin{proof}
Suppose for some $\ol{a} < j$, $\rho_n(w)_{\ol{a},j} = k$ and $\rho_n(w)_{\overline{a}+1,j} = k-1$. Then, by Lemma~\ref{lemma:maximum_length_strictly_decreasing}, there exists a strictly decreasing subsequence $w_{s_1},\dots,w_{s_k}$ of $w$ such that $a \geq w_{s_1} > \dots > w_{s_k} \geq \ol{j} + 1$. 
\par Recall the ``arrow'' notation introduced in Subsection~\ref{subsection:the_stylic_monoid}. We want to show that ${{\uparrow}^{k-1}_w}{(s_k)} = a$. As $w_{s_1},\dots,w_{s_k}$ is a strictly decreasing sequence of length $k$, $b := {{\uparrow}^{k-1}_w}{(s_k)} \neq \varepsilon$. Note that $b \leq a$ as ${{\uparrow}^{l}_w}{(s_k)} \leq w_{s_{k-l}}$ by the definition of ${{\uparrow}^{l}_w}$. Thus, by Lemma~\ref{lemma:up_arrow_to_decreasing_subsequence}, there is a strictly decreasing subsequence $w_{s'_1},\dots,w_{s'_{k-1}},w_{s_k}$ such that $w_{s'_1} = b$. 
However, as $\rho_n(w)_{\overline{a}+1,j} = k-1$, by Lemma~\ref{lemma:maximum_length_strictly_decreasing}, there is no strictly decreasing subsequence of length $k$ only using letters between $\ol{j}+1$ and $a-1$. Hence, $a - 1 < w_{s'_1}$ or $w_{s_k} < \ol{j} + 1$. Thus, $b = w_{s'_1} > a - 1$ as $w_{s_k} \geq \ol{j} + 1$. Therefore, $a = b$, and hence, by Lemma~\ref{lemma:row_bump}, $a$ occurs in the $k$-th row of $N(w)$.
\par Suppose now that $a$ occurs in the $k$-th row of $N(w)$. Hence, by Lemma~\ref{lemma:row_bump}, there exists an index $s_k \leq |w|$ such that ${{\uparrow}^{k-1}_w}{(s_k)} = a$, and therefore, by Lemma~\ref{lemma:up_arrow_to_decreasing_subsequence}, there exists a strictly decreasing subsequence $w_{m_1},\dots,w_{m_k}$ of $w$, where $w_{m_1} = a$ and hence, by Lemma~\ref{lemma:maximum_length_strictly_decreasing}, $\rho(w)_{\ol{a},n+1} \geq k$.

\par Choose $j$ as the  minimum index such that $\rho_n(w)_{\ol{a},j} = k$, which exists by Corollary~\ref{corollary:difference_of_one_in_finite_adjacent_entries_of_matrix}, since $\rho_n(w)_{\ol{a},\ol{a}} = 0$.
Suppose, in order to obtain a contradiction, that $\rho_n(w)_{\overline{a}+1,j} = k$. Let $b < a$ be such that $\rho_n(w)_{\overline{b},j} = k$ and $\rho_n(w)_{\overline{b}+1,j} = \rho_n(w)_{\overline{b},j-1} = k-1$. Notice that such a $b$ exists, by Corollary~\ref{corollary:difference_of_one_in_finite_adjacent_entries_of_matrix}. By Lemma~\ref{lemma:maximum_length_strictly_decreasing}, there exists a strictly decreasing $w_{p_1},\dots,w_{p_k}$ of $w$ such that $b \geq w_{p_1} > \dots > w_{p_k} \geq \ol{j} + 1$. By the same reasoning as given before, we can show that ${{\uparrow}^{k-1}_w}{(p_k)} = b$. Thus, by Lemma~\ref{lemma:up_arrow_to_decreasing_subsequence}, there exists a strictly decreasing subsequence $w_{r_1},\dots,w_{r_k}$ of $w$ such that $w_{r_k} = w_{p_k} \geq \ol{j} + 1$, $w_{r_1} = b$, and ${{\uparrow}^{i-1}_w}{(r_i)} = b$ for $1 < i \leq k$. Notice that, since $\rho_n(w)_{\overline{b},j-1} = k-1$, then $w_{r_k} \leq \ol{j}+1$ by Lemma~\ref{lemma:maximum_length_strictly_decreasing}, otherwise we would have a strictly decreasing subsequence of $w$ of length $k$ only using letters between $\ol{j}+2$ and $\ol{i}$. Hence, $w_{r_k} = \ol{j} + 1$.

On the other hand, as $a$ is in the $k$-th row of $N(w)$, by Lemma~\ref{lemma:row_bump}, there exists $s_k$ such that ${{\uparrow}^{k-1}_w}{(s_k)} = a$, and hence, by Lemma~\ref{lemma:up_arrow_to_decreasing_subsequence}, there exists a strictly decreasing sequence $w_{s_1},\dots,w_{s_k}$ where $w_{s_1} = a$,  ${{\uparrow}^{i-1}_w}{(s_i)} = a$ for $1 < i \leq k$, and $w_{s_k} \leq \ol{j}+1$, since $\rho_n(w)_{\ol{a},j-1} = k-1$.

\par As $w_{s_1} > w_{r_1}$, we have that $r_1 \leq s_1$ as otherwise $w_{s_1},w_{r_1},\dots,w_{r_k}$ would form a strictly decreasing sequence between $a$ and $w_{r_k}$ of length $k+1$. Moreover, if we had $w_{s_2} < w_{r_1}$, then we would have $a = {{\uparrow}^1_w}{(s_2)} \leq w_{r_1} = b < a$ as $r_1 \leq s_2$. Thus, $w_{s_2} \geq w_{r_1}$.

\par By induction, we will show that $w_{s_{i+1}} \geq w_{r_i}$, for all $1 \leq i \leq k-1$. The base case was covered in the previous paragraph. Suppose that there is $1 \leq i \leq k-2$ such that $w_{s_{i+1}} \geq w_{r_i}$. Notice that, if $s_{i+1} < r_{i+1}$, then, by our assumption, $w_{s_{i+1}} \geq w_{r_i} > w_{r_{i+1}}$, and hence $w_{s_1},\dots,w_{s_{i+1}},w_{r_{i+1}},\dots,w_{r_k}$ is a strictly decreasing sequence between $a$ and $w_{r_k}$ of length $k+1$, giving a contradiction. So, $r_{i+1} \leq s_{i+1}$.
Since $w_{s_{i+2}}$ occurs after $w_{s_{i+1}}$, which was shown to occur after $w_{r_{i+1}}$, we have that $w_{s_{i+2}} < w_{r_{i+1}}$ implies ${{\uparrow}^1_w}{(s_{i+2})} \leq w_{r_{i+1}}$ and hence $a = {{\uparrow}^{i+1}_w}{(s_{i+2})} \leq {{\uparrow}^i_w}{(r_{i+1})} = b < a$, giving a contradiction. Thus, $w_{s_{i+2}} \geq w_{r_{i+1}}$. 

\par Therefore, we can conclude that 
\[\ol{j}+1 \geq w_{s_k} \geq w_{r_{k-1}} > w_{r_k} = \ol{j}+1,\]
which results in a contradiction. Thus, $\rho_n(w)_{\ol{a}+1,j} \neq k$, which, by Corollary~\ref{corollary:difference_of_one_in_finite_adjacent_entries_of_matrix}, implies that $\rho_n(w)_{\ol{a}+1,j} = k-1$.
\end{proof}

\begin{theorem}
The morphism $\varrho_n \colon \styln \rightarrow U_{n+1}(\T)$ is a faithful representation of $\styln$.
\end{theorem}
\begin{proof}
It suffices to show that we can construct $N(w)$ from $\rho_n(w)$. By the previous lemma, a letter $a$ is in the $k$-th row of $N(w)$ if and only if there exists an index $j$ such that $\rho_n(w)_{\ol{a}, j} = k$ and $\rho_n(w)_{\ol{a}+1, j} = k-1$. Since $N$-tableaux are uniquely determined by the support of each row (see \cite[Subsection~6.1]{abram_reutenauer_stylic_monoid_2021}), and $\rho_n$ induces $\varrho_n$ by Proposition~\ref{proposition:well_defined_morphism}, we can recover from $\varrho_n([w]_{\styln})$ all the information needed to construct $N(w)$.
\end{proof}

As an example, recall the image of $[4213]_{\styl_4}$ under $\varrho_4$, that is,
\[
    \tikz[tableau]\matrix{
		4 \\
		2 \& 4 \\
		1 \&2 \& 3 \& 4 \\
	};
	\quad \xmapsto{\quad \varrho_4 \quad} \quad
    \begin{bmatrix}
        0 & 1 & 2 & 2 & 3 \\
        -\infty & 0 & 1 & 1 & 2 \\
        -\infty & -\infty & 0 & 1 & 2 \\
        -\infty & -\infty & -\infty & 0 & 1 \\
        -\infty & -\infty & -\infty & -\infty & 0
    \end{bmatrix}
\]
Notice that $\rho_n(4213)_{1,5} = 3$ and $\rho_n(4213)_{2,5} = 2$, hence, $4$ is in the third row of $N(4213)$. However, since $\rho_n(4213)_{2,4} = \rho_n(4213)_{3,4} = 1$ and $\rho_n(4213)_{2,5} = \rho_n(4213)_{3,5} = 2$, we have that $3$ is neither in the second nor the third row of $N(4213)$; on the other hand, since $\rho_n(4213)_{2,3} = 1$, we can conclude that $3$ is in the first row of $N(4213)$. Similarly, we can see that $2$ is in the second, but not the third row, and $1$ is only in the first row. With this information, we have all the necessary information to construct $N(4213)$.

\begin{corollary}
$\cV(\styln) \subseteq \cV(J_n)$ for all $n \in \N$.
\end{corollary}
\begin{proof}
Follows from the previous theorem, and \cite[Corollary 3.3]{johnson_fenner_unitriangular}.
\end{proof}

We now define two semirings: $\N_{\max} := \T \cap (\N  \cup \{0, -\infty\})$; and $[n]_{\max} := \{ x \in \N_{\max} \; | \; x \leq n \}$, for $n \in \N$, with operations $\max$ and $n$-truncated addition, that is $x \otimes y := \min(x+y,n)$. It is easy to see that these semirings are non-trivial and idempotent. Furthermore, we define the morphism $\varphi_{n+1} \colon U_{n+1}(\N_{\max}) \rightarrow U_{n+1}([n]_{\max})$ to be given by $\varphi_{n+1}(X)_{i,j} = \min(X_{i,j},n)$.
\par Note that $\varrho_n(\styln) \subseteq U_{n+1}(\N_{\max})$. For the following corollary, treat $\varrho_n$ as a morphism with codomain $U_{n+1}(\N_{\max})$. Consider the morphism $\ol{\varrho_n} \colon \styln \rightarrow U_{n+1}([n]_{\max})$ defined by $\ol{\varrho_n}([x]_{\styln}) = (\varphi_{n+1} \circ \varrho_n([x]_{\styln}))$, for $x \in \cA_n^*$.
\begin{corollary}
The morphism $\ol{\varrho_n} \colon \styln \rightarrow U_{n+1}([n]_{\max})$ is a faithful representation of $\styln$.
\end{corollary}
\begin{proof}
We can see that $\ol{\varrho_n}$ is a morphism as $\varphi_{n+1}$ and $\varrho_n$ are both morphisms. Moreover, for $w_1,w_2 \in \cA_n^*$, 
\[\ol{\varrho_n}([w_1]_{\styln}) = \ol{\varrho_n}([w_2]_{\styln}) \quad \text{if and only if} \quad \varrho_n([w_1]_{\styln}) = \varrho_n([w_2]_{\styln}),\]
as $\varrho_n([w]_{\styln})_{i,j} \leq n$ for all $1 \leq i,j \leq n+1$ and $w \in \cA_n^*$. Hence, $\ol{\varrho_n}(\styln) \cong \varrho_n(\styln) \cong \styln$.
\end{proof}

\begin{remark}
By \cite{johnson_fenner_unitriangular}, $U_{n+1}(\T), \ U_{n+1}(\N_{\max})$ and $U_{n+1}([n]_{\max})$ satisfy the exact same set of monoid identities. Hence, we gain no more information about the monoid identities satisfied by $\styln$ by considering $\ol{\varrho_n}$ rather than $\varrho_n$.
\end{remark}


\section{Identities of the stylic monoid and connections to other plactic-like monoids}
\label{section:main_result_and_conenctions_to_other_plactic_like_monoids}

We now show that $\styln$ and $U_{n+1}(\T)$ satisfy the exact same set of monoid identities, thus proving that $\styln$ and $U_{n+1}(\T)$ both generate the variety $\Var(J_n)$, and that $\styln$ generates the pseudovariety $\vJ_n$.

\begin{theorem}
Let $n \in \N$ and let $u \approx v$ be a non-trivial identity satisfied by $\styl_n$. Then, $u \approx v \in J_n$.
\end{theorem}
\begin{proof}
We show the contrapositive of the statement. Let $\Sigma = \{x_1,\dots, x_m\}$ be a set of variables, and let $u,v \in \Sigma^*$ be such that $u \approx v \notin J_n$. Without loss of generality, we can assume that there exist variables $a_k,\dots,a_1 \in \Sigma$ which form a subsequence of $u$ but not of $v$, for some $k \leq n$.

Let $y_1,\dots,y_m$ be strictly increasing words over $\cA_k$, defined as follows:
\[ i \in \supp{y_j} \text{ if and only if } a_i = x_j, \]
for $1 \leq i \leq k$, $1 \leq j \leq m$. In other words, $y_j$ is the strictly increasing product of indexes $i$ such that $a_i$ is the variable $x_j$.

Let $\phi \colon \Sigma^* \rightarrow \cA_k^*$ be the evaluation given by $x_i \mapsto y_i$. Notice that, since $j \in \supp{\phi(a_j)}$, then, for any $w \in \Sigma^*$, if $w$ contains the subsequence $a_k,\dots,a_1$, then $\phi(w)$ contains the subsequence $k,\dots,1$. On the other hand, if $\phi(w)$ contains the subsequence $k,\dots,1$, then each index $i$ occurs in some $y_{j_i}$, such that $\phi(w)$ contains the subsequence $y_{j_k},\dots,y_{j_1}$. This implies that $w$ contains the subsequence $x_{j_k},\dots,x_{j_1}$, which, by the definition of $y_j$, is the subsequence $a_k,\dots,a_1$, giving a contradiction. 

Hence, $\phi(u)$ contains the subsequence $k, \dots, 1$, but $\phi(v)$ does not. Therefore, since this subsequence is the only strictly decreasing subsequence, of length $k$, whose first letter is $k$, that can occur in a word over $\cA_k$, we have that, by Lemmas~\ref{lemma:up_arrow_to_decreasing_subsequence} and \ref{lemma:row_bump}, $N(\phi(u))$ contains $k$ in the $k$-th row, but $N(\phi(v))$ does not. Hence $\phi(u) \not\equiv_{\styl} \phi(v)$ and therefore $u \approx v$ is not satisfied by $\styl_k$. Since $k \leq n$, $u \approx v$ is not satisfied by $\styl_n$.
\end{proof}

Therefore, the stylic monoid of rank $n$ joins an increasing list of monoids (see \cite{volkov_reflexive_relations,johnson_fenner_unitriangular}) whose equational theory is $J_n$.

\begin{corollary}
\label{corollary:main_result}
For each $n \in \N$, $\styl_n$ generates the variety $\Var(J_n)$ and the pseudovariety $\vJ_n$. Furthermore, $\Var(\styln) \subsetneq \Var(\styl_{n+1})$, and $\styl_n$ is finitely based if and only if $n \leq 3$: $\styl_1$ admits the basis \ref{basis_J_1}, $\styl_2$ admits the basis \ref{basis_J_2}, and $\styl_3$ admits the basis \ref{basis_J_3}.
\end{corollary}

The following is an immediate consequence of \cite[Section~3]{barker_et_al_scattered_factor_2020}:

\begin{corollary}
For each $n \in \N$, the identity checking problem for $\styln$ is decidable in linearithmic time. Therefore, $\textsc{Check-Id}(\styln)$ is in the complexity class $\mathsf{P}$.
\end{corollary}
\begin{proof}
Taking into account that the number of variables which occur in an identity is less than or equal to the length of the identity, we have that, by the observation after Theorem~8 in \cite[Section~3]{barker_et_al_scattered_factor_2020}, testing if two words over an alphabet of variables share the same subsequences of length $\leq n$ takes $O(l \log l)$ time, where $l$ is the sum of the lengths of both words.
\end{proof}

\begin{corollary}
$\Var(\styln)$ has uncountably many subvarieties, for $n \in \N$ such that $n \geq 3$.
\end{corollary}
\begin{proof}
For any $n \in \N$ such that $n \geq 3$, the word $xyx$ is an isoterm for the equational theory of $\styln$, that is, there is no non-trivial identity $u \approx v$ satisfied by $\styln$, where $u$ or $v$ is the word $xyx$. Hence \cite[Theorem~3.2]{jackson_uncountably_many_subvarieties_2000} applies.
\end{proof}

As such, $\styl_3$ is the only stylic monoid which is simultaneously finitely based and which generates a variety with uncountably many subvarieties. Thus, it is finitely based but not \emph{hereditarily finitely based}, that is, not all of its subvarieties are finitely based. On the other hand, since $\styl_1$ and $\styl_2$ are monoids with a zero and five or less elements, they are hereditarily finitely based \cite{edmunds_lee_continuum}.

We now look at how the stylic monoids relate to other plactic-like monoids, namely the hypoplactic monoid $\hypo$ \cite{novelli_hypoplactic,cm_hypo_crystal}; the sylvester monoid $\sylv$ and the \#-sylvester monoid $\sylvh$ \cite{hivert_sylvester,cm_sylv_crystal}; the Baxter monoid $\baxt$ \cite{giraudo_baxter2,cm_sylv_crystal}; the stalactic monoid $\stal$ \cite{hnt_com_hopf,priez_binary_trees}; the taiga monoid $\taig$ \cite{priez_binary_trees}; the right patience-sorting monoid $\rPS$ \cite{cms_patience}. Let $\boldsymbol{\cC}$ denote the variety of all commutative monoids.  

\begin{corollary}
$\Var(\hypo)$ is the varietal join $\Var(\styl_2) \vee \boldsymbol{\cC}$, and is generated by $\styl_2$ and the free commutative monoid.
\end{corollary}
\begin{proof}
A consequence of \cite[Corollary~4.4]{cain_malheiro_ribeiro_hypoplactic_2022}.
\end{proof}

Notice that $\Var(\styl_3) \not\subseteq \Var(\hypo)$, since $\styl_3$ does not satisfy the identity $xxyx \approx xyxx$: the left-hand side admits $xxy$ as a subsequence, but the right-hand side does not.

\begin{corollary}
$\Var(\sylv)$ and $\Var(\sylvh)$ both strictly contain the varietal join $\Var(\styl_2) \vee \boldsymbol{\cC}$.
\end{corollary}
\begin{proof}
A consequence of the previous corollary, since $\Var(\sylv)$ and $\Var(\sylvh)$ both strictly contain $\Var(\hypo)$ (see \cite{cain_malheiro_ribeiro_sylvester_baxter_2021} for details).
\end{proof}

Notice that $\Var(\styl_3) \not\subseteq \Var(\sylv)$, since $\styl_3$ does not satisfy the identity $xyxy \approx yxxy$: the left-hand side admits $xyx$ as a subsequence, but the right-hand side does not. Similarly, $\Var(\styl_3) \not\subseteq \Var(\sylvh)$, since $\styl_3$ does not satisfy the identity $yxyx \approx yxxy$.

\begin{corollary}
$\Var(\baxt)$ strictly contains the varietal join $\Var(\styl_3) \vee \boldsymbol{\cC}$.
\end{corollary}
\begin{proof}
It is easy to see that $\styl_3$ satisfies the following identities, which form a basis for $\Var(\baxt)$ \cite[Theorem~4.18]{cain_malheiro_ribeiro_sylvester_baxter_2021}:
\begin{align*}
	xzyt \; xy \; rxsy \approx xzyt \; yx \; rxsy; \\
	xzyt \; xy \; rysx \approx xzyt \; yx \; rysx.
\end{align*}
Therefore, $\styl_3 \in \Var(\baxt)$. On the other hand, by \cite[Proposition~6.12]{cain_johnson_kambites_malheiro_representations_plactic-like_2021}, $\Var(\baxt)$ is not contained in the varietal join of $\boldsymbol{\cC}$ and any variety generated by a finite monoid.
\end{proof}

Notice that $\Var(\styl_4) \not\subseteq \Var(\baxt)$, since $\styl_4$ does not satisfy the identity $xzytxyrxsy \approx xzytyxrxsy$: the left-hand side admits $xxyx$ as a subsequence, but the right-hand side does not.

\begin{corollary}
$\Var(\plac_2)$ strictly contains the varietal join $\Var(\styl_3) \vee \boldsymbol{\cC}$.
\end{corollary}
\begin{proof}
A consequence of the previous corollary and \cite[Corollary~4.15]{cain_malheiro_ribeiro_sylvester_baxter_2021} or \cite[Proposition~6.12]{cain_johnson_kambites_malheiro_representations_plactic-like_2021}.
\end{proof}

Notice that $\styl_6 \notin \Var(\plac_2)$, since $\styl_6$ does not satisfy Adjan's identity $xyyx \; xy \; xyyx \approx xyyx \; yx \; xyyx$: the left-hand side admits $xxxyyy$ as a subsequence, but the right-hand side does not. On the other hand, both $\styl_4$ and $\styl_5$ satisfy Adjan's identity, but it is still unclear if $\styl_4$ or $\styl_5$ are in $\Var(\plac_2)$. On a more general note, we have the following upper bound on whether a stylic monoid of finite rank is in the variety generated by a plactic monoid of finite rank:

\begin{corollary}
For any $n \in \N$, $\Var(\placn)$ does not contain $\styl_d$, where $d = 2(2^{\lfloor n^2/4 \rfloor } + \lfloor n^2/4 \rfloor)$.
\end{corollary}
\begin{proof}
Let $\tilde{B}{(2,k-1)}$ be a non-cyclic de Bruijn sequence of order $k-1$ on a size-$2$ alphabet \cite{de_bruijn_sequence}, that is, a word of shortest length in which every possible word of length $k-1$ over a size-$2$ alphabet occurs exactly once as a subsequence. Let $w$ be obtained from $\tilde{B}{(2,k-1)}$ by replacing the first variable by $xy$ and the second by $yx$. The identity 
\[
w \; xy \; w \approx w \; yx \; w
\]
is satisfied by ${UT_k}{(\T)}$ \cite[Section~3.2]{taylor2017upper}. 

Notice that, since a non-cyclic de Bruijn sequence of order $k-1$ on a size-$2$ alphabet is of length $2^{k-1} + k-1$, then $w$ has exactly $2^{k-1} + k-1$ occurrences of each variable. Then, the left-hand side of the identity admits the subsequence $x^{2^{k-1} + k-1} y^{2^{k-1} + n-1}$, but the right-hand side does not. As such, this identity is not satisfied by $\styl_p$, where $p = 2(2^{k-1} + k-1)$. 

Since $\Var(UT_q{(\T)})$ contains $\Var(\placn)$ \cite[Corollary~3.3]{johnson_kambites_tropical_plactic}, where $q = \lfloor n^2/4 \rfloor + 1$, we have that $\Var(\placn)$ does not contain $\styl_d$, where $d = 2(2^{\lfloor n^2/4 \rfloor } + \lfloor n^2/4 \rfloor)$.
\end{proof}

\begin{corollary}
$\Var(\stal) = \Var(\taig)$ and $\Var(\styl_2)$ are incomparable, with respect to inclusion.
\end{corollary}
\begin{proof}
It is easy to see that $\styl_2$ does not satisfy $xyx \approx yxx$, and that $\stal$ does not satisfy $xyxy \approx yxyx$.
\end{proof}

\begin{corollary}
$\Var(\rPSn)$ does not contain $\styl_{n+2}$, for any $n \in \N$.
\end{corollary}
\begin{proof}
By \cite[Proposition~4.7]{cms_patience}, for any $n \in \N$, the right patience-sorting monoid of rank $n$ satisfies the identity $(xy)^{n+1} \approx (xy)^n yx$. The left-hand side of this identity admits $x^{n+1} y$ as a subsequence, but the right-hand side does not. Hence, this identity is not satisfied by $\styl_{n+2}$, and the result follows.
\end{proof}

\section{The finite basis problem for the stylic monoid with involution}
\label{section:The_finite_basis_problem_for_the_stylic_monoid_with_involution}

\par Given a semigroup $S$, an \emph{involution} on $S$ is a unary operation \textup{*} on $S$ such that $(x^*)^* = x$ and $(xy)^* = y^*x^*$. An \emph{involution semigroup} is a semigroup together with an involution, denoted $(S,\textup{*})$. Given an involution semigroup $(S,\textup{*})$, we say the \emph{semigroup reduct} of $(S,\textup{*})$ is the underlying semigroup $S$.

\par The definitions of \emph{involution semigroup variety}, \emph{finitely based involution semigroup}, \emph{identities satisfied by involution semigroups}, and their corresponding involution monoid definitions are defined in an analogous way to the ones given for monoids in Section \ref{section:background}.

\par The unique order-reversing permutation on a finite ordered alphabet $\cA_n$ extends uniquely to an anti-automorphism of the free monoid over $\cA_n$, thus giving an involution. Furthermore, it induces an anti-automorphism of the stylic monoid of rank $n$, which is also an involution (see \cite[Subsection~9.1]{abram_reutenauer_stylic_monoid_2021}). We will denote this involution by $\textup{*}$ and the stylic monoid with involution by $(\styln,\textup{*})$. Similarly, the operation of skew transposition, denoted $^\star$, is an involution on the monoid of unitriangular matrices over the tropical semiring.

We can extend the tropical representation of the stylic monoid of rank $n$ given in Section~\ref{section:tropical_representation_of_the_stylic_monoid} to the involution case:

\begin{proposition} \label{Involution Submonoid}
The morphism $\varrho_n \colon \styln \rightarrow U_{n+1}(\T)$ extends to a faithful morphism from $(\styln,\textup{*})$ to $(U_{n+1}(\T),\,^\star)$.
\end{proposition}
\begin{proof}
It suffices to show that $\varrho_n(x)^\star = \varrho_n(x^*)$ for all $x \in \cA_n$ as then \[\varrho_n(xy)^\star = \varrho_n((xy)^*) = \varrho_n(y^*x^*) = \varrho_n(y)^\star\varrho_n(x)^\star.\] 
For $x \in \cA_n$, we have that $\varrho_n(x^*)_{i,j} = 1$ if and only if $i \leq n+1 - \ol{x} < j$ and $(\varrho_n(x)^\star)_{i,j} = 1$ if and only if $n+1 - j < \ol{x} \leq n+1 -i$. Thus, $\varrho_n(x^*)_{i,j} = 1$ if and only if $(\varrho_n(x)^\star)_{i,j} = 1$, and hence, by the definition of $\varrho_n$, we have that $\varrho_n(x^*) = \varrho_n(x)^\star$.
\end{proof}

In \cite[Section~5]{zhang_upper_tri_trop}, it was shown that the involution monoid $(U_{n+1}(\T),\,^\star)$ is non-finitely based, for $n \geq 3$. It was also shown that $(U_{3}(\T),\,^\star)$ satisfies, for each $k \in \N$, the identity
\[
x y_1 y_1^* y_2 y_2^* \cdots y_k y_k^* x^* z z^* \approx z z^* x y_1 y_1^* y_2 y_2^* \cdots y_k y_k^* x^*.
\]
As such, $(\styl_2,\textup{*})$ must also satisfy these identities. Similarly, it was also shown that $(U_{4}(\T),\,^\star)$ satisfies, for each $k \in \N$, the identity
\[
x_1 x_2 \cdots x_k x_1^* x_2^* \cdots x_k^* x_1 x_2 \cdots x_k \approx x_k^* x_{k-1}^* \cdots x_1^* x_k x_{k-1} \cdots x_1 x_k^* x_{k-1}^* \cdots x_1^*.
\]
As such, $(\styl_3,\textup{*})$ must also satisfy these identities.

However, as with the case of $(U_{n+1}(\B),\,^\star)$ where the involution is again given by skew transposition, we have that $(\styln,\textup{*})$ does not satisfy exactly the same identities as 
$(U_{n+1}(\T),\,^\star)$, in contrast to the monoid reduct case:

\begin{proposition}
For each $n \geq 2$, $(\styln,\textup{*})$ satisfies the identity
\begin{equation}\label{identity:x^*x^n}
x^* x^{n-1} \approx x^* x^{n},
\end{equation}
while $(U_{n+1}(\T),\,^\star)$ does not.
\end{proposition}

\begin{proof}
By \cite[Theorem~5.2]{zhang_upper_tri_trop}, we already know that $(U_{n+1}(\T),\,^\star)$ does not satisfy the identity~(\ref{identity:x^*x^n}). Let $\phi: X \to \styln$ be an evaluation. If $\supp{\phi(x)} = \cA_n$, then $\supp{\phi(x^*)} = \cA_n$ and, as such, each side of the identity~\ref{identity:x^*x^n} has a word representative with a decreasing subsequence of all letters in $\cA_n$. As such, both sides of the identity are equal to $[n \cdots 1]_{\styln} = 0_{\styln}$.
\par On the other hand, suppose $\supp{\phi(x)} \neq \cA_n$. Then, $\phi(x)^{n-1} = \phi(x)^n$, since both elements have a word representative with the maximal decreasing subsequence of elements of its support, of length less than or equal to $n-1$. Equality follows. Therefore, $(\styln,\textup{*})$ satisfies the identity~(\ref{identity:x^*x^n}).
\end{proof}

As such, $(\styln,\textup{*})$ does not generate the same variety as $(U_{n+1}(\T),\,^\star)$, in constrast to the monoid reduct case. It remains open if $(\styln,\textup{*})$ and $(U_{n+1}(\B),\,^\star)$ generate the same variety.

Regarding the question of finite bases for the stylic monoids with involution, it is immediate that $(\styl_1,\textup{*})$ is finitely based, since it is a two-element monoid with a zero. Hence, it admits a finite basis, consisting of the following identities:
\[
x^2 \approx x \quad \text{and} \quad xy \approx yx \quad \text{and} \quad x^* \approx x.
\]

We say an involution semigroup $(S, \textup{*})$ is \emph{twisted} if the variety $\Var(S, \textup{*})$ it generates contains the involution semilattice $(Sl_3, \textup{*})$, where
\[
    Sl_3 = \{ 0, a, b \}
\]
is a semilattice such that $ab=ba=0$ and the involution is given by
\[
0^* = 0, \quad a^* = b, \quad b^*=a.
\]
Notice that any identity $u \approx v$ satisfied by $(Sl_3^1, \textup{*})$, that is, $(Sl_3, \textup{*})$ with an identity adjoined, is such that $\supp{u}=\supp{v}$. In can be easily seen that the variety generated by a twisted involution monoid also contains $(Sl_3^1, \textup{*})$. Therefore, the identities satisfied by any twisted involution monoid must have the same support in both sides of the identity.

\begin{lemma}
For each $n \geq 2$, $(\styln,\textup{*})$ is twisted.
\end{lemma}

\begin{proof}
Consider the quotient of the involution subsemigroup \[ \left\{ [1]_{\styl_2},[2]_{\styl_2},[12]_{\styl_2}, [21]_{\styl_2} \right\} \] of $(\styl_2,\textup{*})$ by the congruence which identifies $[12]_{\styl_2}$ with $[21]_{\styl_2}$. This quotient is isomorphic to $(Sl_3,\textup{*})$, hence $(\styl_2,\textup{*})$ is twisted. Furthermore, since $(\styl_2,\textup{*})$ embeds into $(\styln,\textup{*})$, for each $n \geq 3$, we have that $(\styln,\textup{*})$ is also twisted.
\end{proof}

By \cite[Theorem~4]{lee_non_finitely_based}, we have that any twisted involution semigroup whose semigroup reduct is non-finitely based must also be non-finitely based. Since $\styln$ is non-finitely based for $n \geq 4$, by Corollary~\ref{corollary:main_result}, the following is immediate:

\begin{corollary}
For any $n \geq 4$, $(\styln,\textup{*})$ is not finitely based.
\end{corollary}

Now, we look at the case of $(\styl_2,\textup{*})$:

\begin{proposition}
$(\styl_2,\textup{*})$ is non-finitely based.
\end{proposition}

\begin{proof}
As a consequence of Proposition~\ref{Involution Submonoid}, $(\styl_2,\textup{*})$ satisfies condition (4A) of \cite[Theorem~4.1]{zhang_upper_tri_bool}. As such, we only need to prove that $(\styl_2,\textup{*})$ satisfies condition (4B) to show that it is not finitely based.
\par Notice that \cite[Lemma~4.9]{zhang_upper_tri_bool} holds for $(\styl_2,\textup{*})$: By the same reasoning as the one given in the proof of \cite[Lemma~4.9]{zhang_upper_tri_bool}, and replacing the evaluation used in that proof by any evaluation $\varphi: X \to \styl_2$ such that $\varphi(x) = [1]_{\styl_2}$ gives us the result. Furthermore, \cite[Lemma~4.10]{zhang_upper_tri_bool} also holds for $(\styl_2,\textup{*})$, by taking any evaluation $\varphi: X \to \styl_2$ which maps $x$ and $y$ to $[1]_{\styl_2}$. Thus, $(\styl_2,\textup{*})$ satisfies condition (4B), and is therefore non-finitely based.
\end{proof}

Finally, we look at the case of $(\styl_3,\textup{*})$:

\begin{proposition}
$(\styl_3,\textup{*})$ is non-finitely based.
\end{proposition}
\begin{proof}
As a consequence of Proposition~\ref{Involution Submonoid}, $(\styl_3,\textup{*})$ satisfies condition (5A) of \cite[Theorem~5.2]{zhang_upper_tri_bool}. As such, we only need to prove that $(\styl_3,\textup{*})$ satisfies conditions (5B) and (5C) to show that it is not finitely based.
\par Similarly to the previous case, we have that \cite[Lemmas~5.9--5.12]{zhang_upper_tri_bool} hold for $(\styl_3,\textup{*})$. 
\par In the case of \cite[Lemma~5.9]{zhang_upper_tri_bool}, consider the involution submonoid $Z' = \{ [\varepsilon]_{\styl_3}, [123]_{\styl_3}, [2312]_{\styl_3}, [321]_{\styl_3} \}$. Notice that $[123]_{\styl_3}^2 = [2312]_{\styl_3}$. The proof then follows from \cite[Lemma~6.1.4]{almeida_1994_finite_semigroups}.
\par For \cite[Lemma~5.10]{zhang_upper_tri_bool}, consider the evaluations $\varphi: X \to \styl_3$ which maps $x$ to $[23]_{\styl_3}$ and $\phi: X \to \styl_3$ which maps $x$ to $[12]_{\styl_3}$. Notice that $[23]^*_{\styl_3} = [12]_{\styl_3}$. Then, using the reasoning in the proof of \cite[Lemma~5.10]{zhang_upper_tri_bool}, the result follows. Notice that it is also relatively easy to see that $[u u^* u]_{\styl_3}= [u^* u u^*]_{\styl_3}$.
\par For \cite[Lemma~5.11]{zhang_upper_tri_bool}, consider the evaluations $\varphi: X \to \styl_3$ which maps $x$ to $[3]_{\styl_3}$ and $y$ to $[2]_{\styl_3}$ and $\phi: X \to \styl_3$ which maps $x$ to $[1]_{\styl_3}$ and $y$ to $[2]_{\styl_3}$. Replace the first evaluation in the proof of \cite[Lemma~5.11]{zhang_upper_tri_bool} by $\varphi$ and the second by $\phi$, and the image of $w$ by both $\varphi$ and $\phi$ will be different from $[321]_{\styl_3}$. The result follows and $(\styl_3,\textup{*})$ satisfies condition (5B).
\par Finally, for \cite[Lemma~5.12]{zhang_upper_tri_bool}, we consider the substitutions $\varphi: X \to \styl_3$ which maps $t$ to $[3]_{\styl_3}$ and $s$ to $[2]_{\styl_3}$ and $\phi: X \to \styl_3$ which maps $t$ to $[12]_{\styl_3}$ and $s$ to $[21]_{\styl_3}$, and the result follows. Hence, $(\styl_3,\textup{*})$ satisfies condition (5C). As such, $(\styl_3,\textup{*})$ is non-finitely based.
\end{proof}

Therefore, we obtain the following result:

\begin{theorem}
The involution monoid $(\styln,\textup{*})$ is finitely based if and only if $n=1$. 
\end{theorem}

\bibliographystyle{plain}
\bibliography{stylic_monoid_identities.bib}
\end{document}